\documentclass[a4paper,12pt]{amsart}
\usepackage{cmap}

\usepackage[T2A]{fontenc}
\usepackage[utf8]{inputenc}

\usepackage{amssymb,latexsym,amsmath,amsfonts,bbm,amsthm,amscd,amsbsy}
\usepackage{amsthm}
\usepackage{hyperref}

\usepackage{url}

\makeatletter 
\@addtoreset{equation}{subsection}
\makeatother

\DeclareMathOperator{\rk}{rk}
\DeclareMathOperator{\Aut}{Aut}

\DeclareMathOperator{\Int}{Int}
\DeclareMathOperator{\Hom}{Hom}
\DeclareMathOperator{\Cl}{Cl}
\DeclareMathOperator{\GL}{GL}
\DeclareMathOperator{\Div}{Div}

\newcommand{\Z}{\mathbb{Z}}
\renewcommand{\C}{\mathbb{C}}
\renewcommand{\P}{\mathbb{P}}
\newcommand{\Q}{\mathbb{Q}}
\newcommand{\CC}{\mathfrak{C}}
\renewcommand{\S}{\mathfrak{S}}

\newtheorem{proposition}[subsection]{Proposition}

\newtheorem{theorem}[subsection]{Theorem}
\newtheorem{corollary}[subsection]{Corollary}
\newtheorem{lemma}[subsection]{Lemma}

\theoremstyle{definition}
\newtheorem{definition}[subsection]{Definition}
\swapnumbers

\newtheoremstyle{dotless}{}{}{\rm}{}{\sc}{}{ }{}

\theoremstyle{dotless}
\newtheorem{case}[subsection]{Case:}
\newtheorem{scase}[equation]{Subcase:}

\title{Toric $G$-Fano threefolds}
\author{Arman Sarikyan}
\address{\textnormal{School of Mathematics, The University of Edinburgh, Edinburgh, UK.}
\newline
\textnormal{\url{ar.sarikyan@gmail.com}}}
\date{}

\begin{document}
\maketitle
	\begin{abstract}
We classify toric Fano threefolds having at worst terminal singularities such that the rank of the $G$-invariant part of the class group equals one, where $G$ is a finite group acting on the variety by automorphisms.
	\end{abstract}

\section{Introduction.}	
	Let $X$ be an algebraic variety over the field of complex numbers $\C$ and let $G$ be a finite group. We say that $X$ is a $G$-{\itshape variety} if the action of $G$ on $X$ is defined by a homomorphism $G\to \Aut(X)$.
	
	A $G$-variety $X$ is called $G$-{\itshape Fano} variety if singularities of $X$ are not worse than terminal, the anticanonical class $-K_X$ is ample and the rank of the $G$-invariant part of the class group $\Cl(X)^G$ equals one. Such varieties appear naturally in the study of many problems of the equivariant birational geometry (see \cite{Prokhorov2009e}, \cite{cheltsov2020toric}, \cite{cheltsov2015cremona}). Partial results on the classification of the $G$-Fano varieties were obtained in \cite{Prokhorov-GFano-1}, \cite{Prokhorov-GFano-2}, \cite{Prokhorov-planes}. 
	
	A \textit{toric variety} is an algebraic variety $X$ containing a torus $T\cong (\C^*)^n$ as a dense Zariski open subset such that an action of $T$ on itself extends to an algebraic action on $X$. 
	
	In this work we classify the toric $G$-Fano threefolds. The main result obtained is the following.

	\begin{theorem}\label{main}
		Let $X$ be a toric $G$-Fano threefold. Then there are the following possibilities:
		\begin{center}
			\begin{tabular}{|c|c|c|c|c|}
				\hline
				$X$ & $\rho(X)$ & $(-K_X)^3$ & g &\#\\ \hline
				$\P^3$        & $1$  & $64$ & $33$ & $4$ \\ \hline
				$\P^3/\boldsymbol{\mu}_5$ & $1$  & $64/5$ & $5$ & $1$ \\ \hline
				$\P(1,1,1,2)$ & $1$  & $125/2$ & $32$ & $7$ \\ \hline
				$\P(1,1,2,3)$ & $1$  & $343/6$ & $29$ & $8$ \\ \hline
				$\P(1,2,3,5)$ & $1$  & $1331/30$ & $22$ & $6$ \\ \hline
				$\P(1,3,4,5)$ & $1$  & $2197/60$ & $18$ & $5$ \\ \hline
				$\P(2,3,5,7)$ & $1$  & $4913/210$ & $11$ & $2$ \\ \hline
				$\P(3,4,5,7)$ & $1$  & $6859/420$ & $7$ & $3$ \\ \hline
				\eqref{pyr1} & $1$ & $343/12$ & $14$ & $19$\\ \hline
				The quadric cone in $Q\subset\P^4$ with one singular point & $1$ & $54$ & $28$ & $32$ \\ \hline
				\eqref{pyr3} & $1$ & $125/3$ & $21$ & $33$\\ \hline
				\eqref{81/2} & $3$ & $81/2$ & $21$ & $92$\\ \hline
				$\P^1 \times \P^1 \times \P^1$ & $3$ & $48$ & $25$ & $62$ \\ \hline
				$(\P^1 \times \P^1 \times \P^1)/\boldsymbol{\mu}_2$ & $3$ & $24$ & $11$ & $47$ \\ \hline
				The divisor of type $(1,1,1,1)$ in $(\P^1)^4$ & $4$ & $24$ & $13$ & $625$ \\ \hline
				The complete intersection of two quadrics $V_{2,2} \subset \P^5$ & $5$ & $32$ & $17$ & $297$ \\ \hline
		
			\end{tabular}
		\end{center}
		where $g:=\dim |-K_X|-1$ is the genus of Fano variety $X$, $\rho(X)$ is the rank of the Picard group of $X$,  $\#$ is the number in \cite{GRD}, $(\P^1)^n$ denotes the product of $n$ copies of $\P^1$ and $\boldsymbol{\mu}_2$ is generated by the involution that acts diagonally and faithfully on each factor.
	\end{theorem}
	
	All the toric Fano threefolds having at worst terminal singularities were classified by A. Kasprzyk. The classification is given purely in terms of convex polytopes associated with the varieties. One can find the classification algorithm of A. Kasprzyk in \cite{kasprzyk2006toric} and its output is available in \cite{GRD}.
	
	{\bfseries Acknowledgments.} I am grateful to Yu. Prokhorov for introducing me to this topic and for invaluable advice. I am also grateful to C. Shramov for useful conversations, to I. Cheltsov for pointing me out a gap in the earlier version of Lemma \ref{gap} and to the referee for useful comments. This work was supported by the Foundation for the Advancement of Theoretical Physics and Mathematics ``BASIS''.
	
	\section{Preliminaries.}
	
	\subsection{Notation} Everywhere below we use the following notation.
	\begin{itemize}
		\item $N$ is a lattice of one-parameter subgroups of a torus; 
		\item $M$ is a character lattice of a torus, i.e. $M=\Hom(N,\mathbb{Z})$;
		\item $N_\Q:= N \otimes \Q$;
		\item $M_\Q:= M \otimes \Q$; 
		\item $\Sigma$ is a fan in $N_\Q$;
		\item $X_{\Sigma}$ is a toric variety associated with $\Sigma$;
		\item $\Div_T(X_\Sigma)$ is the group of $T$-invariant Weil divisors on $X_\Sigma$;
		\item $\Div_T(X_\Sigma)_\Q := \Div_T(X_\Sigma) \otimes \Q$;
		\item $\Cl(X_\Sigma)_\Q:=\Cl(X_\Sigma) \otimes \Q$;
		\item $v(P)$ is the number of vertices of $P$;
		\item $\Int(P)$ is the set of the interior $N$-valued points of $P$;
		\item $\Aut(X_\Sigma)$ is the group of automorphisms of $X_{\Sigma}$;
		\item $\CC_n$ is the cyclic group of order $n$;
		\item $\S_n$ is the symmetric group on $n$ elements;
		\item $V/\boldsymbol{\mu}_n$ is a quotient of a toric variety $V$ by a cyclic group of order $n$ defined by the torus quotient.
	\end{itemize}
	
	\subsection{Divisors on the Toric Varieties.} 
	Denote by $\Sigma(n)$ the set of $n$-dimensional subcones of $\Sigma$. Then every one dimensional subcone (ray) $r \in \Sigma(1)$ is associated with a $T$-invariant Weil divisor $D_r$. Moreover, every Weil divisor on $X_\Sigma$ is linearly equivalent to some $T$-invariant divisor. Thus,
	$$
	\Div_T(X_\Sigma)= \bigoplus \limits_{r \in \Sigma(1)} \Z D_r.
	$$
	
	\begin{theorem}[{\cite[Theorem 4.1.3]{cox2011toric}}]
		For any projective toric variety $X_{\Sigma}$ the following sequence is exact:
		\begin{equation}\label{exact}
		\begin{CD}
		0 @>>> M @>>> \Div_T(X_\Sigma) @>>> \Cl(X_\Sigma) @>>> 0.
		\end{CD}
		\end{equation}
	\end{theorem}
	
	\subsection{Toric Fano Varieties.}
	
	\begin{definition}
		A convex lattice polytope $P \subset N_\Q$ is a \textit{Fano polytope} if $O \in \Int(P)$, where $O$ is the origin of $N$, and the vertices of $P$ are all primitive elements of $N$. 
	\end{definition}
	
	\begin{definition}[{\cite[Definition 3.7.2]{kasprzyk2006toric}}]
		Let $P$ be an $n$-dimensional Fano polytope.
		\renewcommand\labelenumi{\rm (\arabic{enumi})}
		\renewcommand\theenumi{\rm (\arabic{enumi})}
		\begin{enumerate}
			\item 
			If the vertices of any facet of $P$ form a $\Z$-basis of $N$, then $P$ is called a \textit{regular} (or smooth) Fano polytope;
			\item 
			If every facet of $P$ is an $(n-1)$-simplex, then $P$ is called a \textit{simplicial} (or $\Q$-factorial) Fano polytope;
			\item 
			If the dual polytope $P^* \subset M_{\Q}$ is also a lattice polytope, then $P$ is called a \textit{reflexive} (or Gorenstein) Fano polytope;
			\item 
			If the only lattice points on or in $P$ consist of the vertices and the origin, then $P$ is called a \textit{terminal} Fano polytope;
			\item 
			If the only interior lattice point of $P$ is the origin, then $P$ is called a \textit{canonical} Fano polytope.
		\end{enumerate}
	\end{definition}
	
	There is a correspondence between isomorphism classes of toric Fano varieties and isomorphism classes of Fano polytopes \cite{kasprzyk2006toric}.

	\begin{lemma}\label{polygons}
		Up to the action of $\GL(2,\Z)$ there are exactly two convex lattice polygons $V \subset \Z^2$ such that there are no lattice points on the boundary of $V$ except the vertices and there are no interior lattice points. These are a triangle and a convex quadrangle.
	\end{lemma}
	
	\begin{proof}
		Without loss of generality, we can assume that the points $(0,0)$,  $(1,0)$, $(0,1)$ are consecutive vertices of $V$. Suppose $V$ is not a triangle. Then the point $(1,1)$ must be the vertex of $V$, otherwise, due to the convexity, it would be an interior point of $V$. Hence, again due to the convexity of $V$, the consecutive vertex must lie on the line $(0,n)$ for $n\in \Z$. But there no lattice points on the boundary of $V$, thereby $n=1$.
	\end{proof}
	
	\begin{corollary} \label{facets_corollary}
		If $P$ is a $3$-dimensional terminal Fano polytope, then $P$ can have only triangle or quadrangle facets.
	\end{corollary}
	
	\begin{lemma}[{\cite[Lemma 1]{Batyrev1981}}]	\label{lemma_bat}
		If $P$ is a $3$-dimensional terminal Fano polytope, then $v(P)\leq 14$.
	\end{lemma}
	
	Note that the terminal Fano polytopes with $14$ vertices do exist.
	
	Recall that {\itshape Archimedean solid} is a convex solid composed of regular polygons.
	
	\begin{theorem}[\cite{Gruenbaum1984}]\label{grum}
		Let $P$ be a convex $3$-dimensional polytope, $G$ be a group acting transitively on the vertices of $P$. Then $P$ is combinatorially equivalent to either a Platonic solid, an Archimedean solid, a prism or an antiprism.
	\end{theorem}
	
	\begin{proposition}[\cite{Tahara}]\label{orders}
		Let $G\subset \GL(3,\Z)$ be a finite subgroup. Then the order of $G$ divides $48$. Hence, if $n$ is an integer dividing $\left|G\right|$ such that $5\le n \le 14$ , then $n = 6$, $8$ or $12$. 
	\end{proposition}
	
	\begin{proposition}\label{arch}
		Let $n= 6,8$ or $12$. Then there are exactly two Archimedean solids with $n$ vertices: the truncated tetrahedron and the cuboctahedron. Both of them have $12$ vertices.   
	\end{proposition}
	
	\begin{theorem}[\cite{Bor}, \cite{Kasprzyk2006}]\label{rank_one}
		Let $X$ be a toric Fano threefold having at worst terminal singularities with $\rk \Cl (X)=1$. Then $X$ is either $\P^3$, the quotient of $\P^3/\boldsymbol{\mu}_5$ with weights  $(1,2,3,4)$, or a weighted projective space $\P(w_0,\dots,w_3)$ with weights 
		$$
		(1,1,1,2),\, (1,1,2,3),\, (1,2,3,5),\, (1,3,4,5),\, (2,3,5,7),\, (3,4,5,7).
		$$
	\end{theorem}
	
	\section{Toric G-Fano Threefolds.}
	
	From now on, let $P$ be a $3$-dimensional terminal Fano polytope and $N\cong \Z^3$ (hence $M\cong \Z^3$). Let also $\Sigma$ be a fan associated with $P$ and $X_\Sigma$ be the associated with $\Sigma$ toric $G$-Fano threefold. 
	
	\begin{proposition} 
	In the notation above,  $\rk\Cl(X_{\Sigma})^G=1$ if and only if there exists a subgroup $H\subset N(T)$ such that $\rk\Cl(X_{\Sigma})^{H}=1$, where $N(T)$ is the normalizer of $T$ in $\Aut(X_{\Sigma})$.
	\end{proposition}
	\begin{proof}
		Let $\mathfrak{M}$ be a set of maximal tori in $\Aut^0(X_\Sigma)$, where $\Aut^0(X_\Sigma)$ is the connected identity component of $\Aut(X_\Sigma)$. The group $\Aut^0(X_\Sigma)$ acts on $\mathfrak{M}$ by conjugation and this action is transitive since $\Aut^0(X_\Sigma)$ is connected by \cite[Corollary 11.3]{Borel1991}. Thereby, for each element $g\in G$ there exists an element $h\in \Aut^0(X_\Sigma)$ such that $ghT(gh)^{-1} = T$. Now denote by $H$ the group generated by all elements $gh$. Since the action of $\Aut^0(X_\Sigma)$ on $\Cl(X_\Sigma)$ is trivial and the images of $G$ and $H$ coincide in $\Aut(X_{\Sigma})/\Aut^0(X_{\Sigma})$, we are done. 	
	\end{proof}	
	Consequently, we can assume that $G$ is a subgroup of $N(T)$. Also, for $N(T)$ we have the following split exact sequence:
	$$
	\begin{CD}
	1 @>>>T @>>> N(T) @>\nu>> W@>>> 1,
	\end{CD}
	$$
	where $W$ is a finite subgroup of $\GL(3,\Z)$, known as the Weyl group of $X_{\Sigma}$. Since the action of $G$ on $\Cl(X_{\Sigma})$ depends only on the image of $\nu$, we will assume that $G$ is a subgroup of $W$.
	
	Consider the exact sequence of vector spaces obtained from (\ref{exact}) by tensoring by $\Q$:
	$$
	\begin{CD}
	0 @>>>M\otimes \Q @>>> \Div_T(X_\Sigma)\otimes \Q @>>> \Cl(X_\Sigma) \otimes \Q @>>> 0.
	\end{CD}
	$$
	Taking the $G$-invariants gives us an isomorphism
	$$
	\Div_T(X_\Sigma)_\Q^G /M_\Q^G \cong \Cl(X_\Sigma)_\Q^G.
	$$
	
	Now we are ready to prove Theorem \ref{main}.
	We prove it by considering the terminal Fano polytopes that correspond to the toric terminal Fano threefolds. The proof will be divided into four lemmas.
	
	\begin{lemma}\label{lemma1}
		Let $\dim\Div_T(X_\Sigma)_\Q^G=4$, $\dim M_\Q^G=3$. Then $X_\Sigma$ is one of the following:
		\begin{itemize}
			\item[\textup{1)}] $\P^3$;
			\item[\textup{2)}] the quotient $\P^3/\boldsymbol{\mu}_5$ with weights  $(1,2,3,4)$;
			\item[\textup{2)}] the weighted projective space $\P(w_0,\dots,w_3)$ with weights 
		$$
		(1,1,1,2),\, (1,1,2,3),\, (1,2,3,5),\, (1,3,4,5),\, (2,3,5,7),\, (3,4,5,7);
		$$
		\end{itemize}
		and $G$ acts trivially.
	\end{lemma}
	\begin{proof}
		By assumption the group $G$ has three invariant $1$-dimensional subspaces in $N_\Q$, hence it acts trivially on $\Div_T(X_\Sigma)_\Q$. Consequently, there are exactly four rays in $\Sigma$. One can easily see that in this case $\rk \Cl (X_\Sigma)=1$, so we are done by Theorem \ref{rank_one}.
	\end{proof}
	From now on, we will assume that $\rk\Cl (X_\Sigma)>1$, i.e. $v(P)>4$.
	\begin{lemma}\label{lemma}
		Let $\dim\Div_T(X_\Sigma)_\Q^G=3$, $\dim M_\Q^G=2$ and $\rk\Cl (X_\Sigma)>1$. Then $X_{\Sigma}$ is one of the following:
			\begin{itemize}
				\item[\textup{1)}] the terminal Fano threefold of degree $343/12$ and genus $14$;
				\item[\textup{2)}] the quadric cone $Q\subset\P^4$;
				\item[\textup{3)}] the terminal Fano threefold of degree $125/3$ and genus $21$,
			\end{itemize}
		and $G\cong\CC_2$.
	\end{lemma}
	\begin{proof}
		By assumption the group $G$ has two invariant $1$-dimensional subspaces in $N_\Q$, so $G$ acts by reflection with respect to the plane $\pi$ generated by the invariant $1$-dimensional subspaces. Hence  $v(P)\leq 6$, since $\dim \Div_T(X_\Sigma)_\Q^G = 3$. Consider the following cases.
		
		\begin{case} {\sc there are no vertices of $P$ on $\pi$.}\label{cyl}
			Then $v(P)$ is even and hence $v(P)=6$. So, $P$ is a prism symmetric with respect to $\pi$ with triangle bases. By \cite{GRD} there is only one terminal Fano prism with triangle bases having the following vertices:
			$$
			(1,0,0),\, (0,1,0),\, (1,1,1),\, (-1,-1,0),\, (0,0,-1),\, (-2,-1,-1).
			$$
			However, this prism does not admit a group action satisfying our conditions.
   		\end{case}
   		
   		\begin{case} {\sc there is one vertex of $P$ on $\pi$.}\label{pyramids}
			Obviously, $v(P)=5$ and hence $P$ is a pyramid which top is on $\pi$ and its base is a convex lattice quadrangle symmetric with respect to $\pi$. There are three terminal Fano pyramids having the convex quadrangle base by \cite{GRD}:
			\begin{itemize}
				\item The pyramid with the vertices 
					\begin{equation}\label{pyr1}
					(1,0,0),\, (0,1,0),\, (1,-2,3),\, (-1,1,-2),\, (0,-2,1).	
					\end{equation} 
					This polytope determines the terminal Fano threefold of degree $343/12$ and genus $14$ (number $19$ in \cite{GRD}). The group $G\cong\CC_2$ is generated by:
						$$
							\begin{pmatrix}
							1 & -1 & -1\\
							0 & 1 & 0\\
							0 & -2 & -1\\
							\end{pmatrix}.
						$$
				
				\item The pyramid with the vertices
					\begin{equation*}
					(1,0,0),\, (0,1,0),\, (1,1,1),\,(-1,-1,0),\,(0,0,-1).
					\end{equation*}
				This polytope determines the quadric cone $Q\subset \P^4$ (number $32$ in \cite{GRD}). The group $G\cong\CC_2$ is generated by:
					$$
						\begin{pmatrix}
						1 & -1 & -1\\
						0 & 1 & 0\\
						0 & -2 & -1\\
						\end{pmatrix}.
					$$

				\item The pyramid with the vertices 
					\begin{equation}\label{pyr3}
						(1,0,0),\,(0,1,0),\,(1,1,2),\,(-1,-1,-1),\,(0,-1,1).
					\end{equation}			
					This polytope determines the terminal Fano threefold of degree $125/3$ and genus $21$ (number $33$ in \cite{GRD}). The group $G\cong\CC_2$ is generated by:
					$$
						\begin{pmatrix}
						1 & -1 & 0\\
						0 & -1 & 0\\
						0 & -1 & 1\\
						\end{pmatrix}.
					$$
			\end{itemize} 
			
   		\end{case}

		\begin{case}{\sc there are two vertices of $P$ on $\pi$.}
			Then every vertex on $\pi$ gives us the orbit. So we have only one additional orbit of order two. Thus $v(P)=4$ and we are in the case of $\rk \Cl (X_\Sigma) =1$.
	    \end{case}	
		
		The plane $\pi$ cannot contain more than two vertices, otherwise $\dim M_\Q^G \geq 3$.
	\end{proof}
	
	\begin{lemma}\label{lem}
	Let $\dim\Div_T(X_\Sigma)_\Q^G=2$, $\dim M_\Q^G=1$ and $\rk\Cl (X_\Sigma)>1$. Then $X_\Sigma$ is one of the following:
	\begin{itemize}
	\item[\textup{1)}] the quadric cone $Q\subset \P^4$ and $G\cong \CC_4$;
	\item[\textup{2)}] the terminal toric Fano threefold of degree $81/2$ and $G\cong \CC_3$;	
	\item[\textup{3)}] $\P^1\times\P^1\times\P^1$ and $G\cong \CC_3$;
	\item[\textup{4)}]$\P^1\times\P^1\times\P^1/\boldsymbol{\mu}_2$ and $G\cong \CC_3$, where $\boldsymbol{\mu}_2$ is generated by the involution that acts diagonally and faithfully on each factor; 
	\item[\textup{5)}] the complete intersection of two quadrics $V_{2,2}\subset\P^5$ and $G\cong \CC_4$.
	\end{itemize}

	\end{lemma}
	\begin{proof}
	By assumption the group $G$ has one invariant $1$-dimensional subspace in $M_\Q$ and the vertices of $P$ are contained in two parallel planes, otherwise there will be more than $2$ orbits of the vertices, which contradicts with $\dim \Div_T(X_{\Sigma})_{\Q}=2$. Denote by $F_1$, $F_2$ the facets of $P$ contained in these parallel planes. Then $v(F_i)\le 4$ by Lemma \ref{polygons}. Skipping the cases when $P$ is a terminal Fano tetrahedron, i.e. when $\rk \Cl(X_\Sigma)=1$, we have the following cases.
	
	\begin{case}{\sc $F_1$ is a point, $F_2$ is a quadrangle.} In this case $P$ is a pyramid with the quadrangle base. We already noticed in Lemma \ref{lemma}  Case \ref{pyramids} that there are three such pyramids. Straightforward computations show that there is only one pyramid which admits the desired group action. More precisely, it is the pyramid with the following vertices: 
	    $$
			(1,0,0),\, (0,1,0),\, (1,1,1),\,(-1,-1,0),\,(0,0,-1).
		$$
		Recall that this polytope determines the quadratic cone $Q$ in $\P^4$ with one singular point. The group $G\cong\CC_4$ can be generated by
	    $$
			\begin{pmatrix}
			1 & 0 & -1\\
			1 & 0 & 0\\
			1 & -1 & 0\\
			\end{pmatrix}.
		$$
	\end{case}

	\begin{case}{\sc $F_1$ is a segment, $F_2$ is a quadrangle.} In this case $P$ is the prism with the triangle bases. Notice that such polytope does exist by \cite{GRD}, it is unique and appeared in Lemma \ref{lemma} Case \ref{cyl} but now it has the desired group action. Recall that the vertices of $P$ are 
			\begin{equation}\label{81/2}
			(-1,-1,0),\, (1,0,0),\, (1,1,1),\, (-2,-1,-1),\, (0,0,-1),\, (0,1,0).
			\end{equation}
			
			This polytope determines the terminal toric Fano threefold of degree $81/2$ and genus $21$ (number $92$ in \cite{GRD}). The group $G\cong\CC_3$ can be generated by 
			$$
			\begin{pmatrix}
			-1 & 0 & 2\\
			-1 & 0 & 1\\
			0 & -1 & 1\\
			\end{pmatrix}.
			$$	

	\end{case}
	
	\begin{case}{\sc $F_1$ and $F_2$ are the triangles.}\label{octahedron}
	In this case  $P$ is a triangle prism or a triangle antiprism. The case of triangle prism was considered above.
	
	 Let $P$ be the triangle antiprism, which is the same as octahedron. Then $G$ cyclicly permutes all vertices of each facet, so must be isomorphic to $\CC_3$ and $F_1$, $F_2$ are equilateral lattice triangles, with respect to the combinatorial lattice distance. By \cite[Proposition 3]{Tahara}, up to conjugation, there are only two subgroups of order $3$ in $\GL(3,\Z)$:
		$$
			W_1=\left\langle
			\begin{pmatrix}
			1 & 0 & 0\\
			0 & 0 & -1\\
			0 & 1 & -1\\
			\end{pmatrix}\right\rangle,\
			W_2=\left\langle
			\begin{pmatrix}
			0 & 1 & 0\\
			0 & 0 & 1\\
			1 & 0 & 0\\
			\end{pmatrix}\right\rangle.
		$$
	Consider the following subcases.
	\begin{scase}{\sc $G=W_1$.}
	Then the facets are contained in the planes of the form $x=n$ for some integer $n$. So, the vertices of the facet have the following coordinates:
	$$
	(n,b,c),\,(n,-c,b-c),\,(n,-b+c,-b),
	$$
	where $b$ and $c$ are some integers. Thus, on the one hand, the area of the facet is equal to 
	$$
	\dfrac{1}{2}\left| 3b^2-3bc+3c^2\right|
	$$
	but on the other hand, it should be equal to $1/2$ which is absurd.
	\end{scase}
	
	\begin{scase}{\sc $G=W_2$.} 
	Then the centers of the facets are on the line generated by the vector $(1,1,1)$, and also $F_1$, $F_2$ are contained in the planes of the form $x+y+z=n$ for $n=\{-2,-1,1,2\}$. So, the vertices of the facet are
	$$
	(a,b,n-a-b),\,(b,n-a-b,a),\,(n-a-b,a,b).
	$$
	Hence, we have the following equation for the area of the facet 
	$$
	\dfrac{\sqrt{3}}{2}\left| 3a^2+3b^2+3ab-3an-3bn+n^2\right|=\dfrac{\sqrt{3}}{2}.
	$$
	Straightforward computations show that we have only the following solutions:
	\begin{itemize}
		\item $n=2,\ a=1,\, b=1$;
		\item $n=1,\, a=1,\, b=0$;
		\item $n=-1,\, a=-1,\, b=0$;
		\item $n=-2,\, a=-1,\, b=-1$.
	\end{itemize}
	Forming from the triangles defined by the solutions above the triangle antiprisms we get the following:
	\begin{itemize}
		\item the polytope with the vertices 
			$$
			\pm(1,0,0),\, \pm(0,1,0),\, \pm(0,0,1),
			$$
			which determines $\P^1\times\P^1\times\P^1$ (number $62$ in \cite{GRD});
		\item the polytope with the vertices 
			$$
			\pm (1,1,0),\, \pm (1,0,1),\, \pm (0,1,1),
			$$
			which determines $\P^1\times\P^1\times\P^1/\boldsymbol{\mu}_2$ (number $47$ in \cite{GRD}).	
	\end{itemize}
	\end{scase}
	\end{case}
	
	\begin{case}{\sc $F_1$ and $F_2$ are quadrangles.}\label{cube}
	In this case $P$ is a quadrangle prism or a quadrangle antiprism. But there are no terminal Fano quadrangle antiprisms by \cite{GRD}. Let $P$ be terminal Fano quadrangle prism. By \cite{GRD} such prism exists, it is unique and its vertices are
	$$
	\pm (1,0,0),\, \pm(0,0,1),\, \pm(1,1,1),\, \pm(0,1,0).
	$$
	This polytope determines the intersection of two quadrics $V_{2,2}\subset\P^5$ (number $297$ in \cite{GRD}) having six isolated singular points, which can be given by the equations (see \cite[(7.5)]{Prokhorov-GFano-1})
	$$
	x_0^2-x_1^2=x_2^2-x_3^2=x_4^2-x_5^2.
	$$
    The group $G\cong \CC_4$ can be generated by
			$$
			\begin{pmatrix}
			0 & 1 & -1\\
			-1 & 1 & 0\\
			0 & 1 & 0\\
			\end{pmatrix}.
			$$
	
	\end{case}
	\end{proof}
	
	\begin{lemma}\label{gap}
		Let $\dim\Div_T(X_\Sigma)_\Q^G=1$, $\dim M_\Q^G=0$ and $\rk\Cl (X_\Sigma)>1$. Then $X_\Sigma$ is one of the following:
		\begin{itemize}
	\item[\textup{1)}] $\P^1\times\P^1\times\P^1$;
	\item[\textup{2)}] $\P^1\times\P^1\times\P^1/\boldsymbol{\mu}_2$;
	\item[\textup{3)}] the complete intersection of two quadrics $V_{2,2}\subset\P^5$;
	\item[\textup{4)}] the divisor of type $(1,1,1,1)$ in $(\P^1)^4$,
	\end{itemize}
	
	and $G\cong\S_4\times\CC_2$.
	\end{lemma}
	\begin{proof}
	By assumption the group $G$ has no invariant subspaces in $M_\Q$ and acts transitively on the vertices of $P$. By Lemma \ref{lemma_bat} we know that $v(P)\leq 14$. Hence, $5\leq v(P)\leq 14$. Thereby, by Proposition \ref{orders} we conclude that $v(P)=6,\,8,\,12$. Combining together Corollary \ref{facets_corollary} and Theorem \ref{grum} we have the following cases.
	
	\begin{case} {\sc  $P$ combinatorially is a Platonic solid.}
	
	\begin{scase}{\sc $v(P)=6$.} 
	Then $P$ combinatorially is an octahedron. Since $G$ acts transitively, it contains a subgroup of order $3$ which cyclicly permutes the vertices of the triangle facet. But the octahedrons with such group action were already considered in Lemma \ref{lem} Case \ref{octahedron}. Thus, we have no new varieties in this subcase. However, the described tetrahedrons also admit the group actions with $\dim\Div_T(X_\Sigma)_\Q^G=1$ and $\dim M_\Q^G=0$. More precisely, for $\P^1\times\P^1\times\P^1$ the group $G$ can be generated by
	$$
		\begin{pmatrix}
			0 & 0 & 1\\
			0 & 1 & 0\\
			-1 & 0 & 0\\
		\end{pmatrix},\,
		\begin{pmatrix}
			-1 & 0 & 0\\
			0 & 0 & -1\\
			0 & -1 & 0\\
		\end{pmatrix},\,
		\begin{pmatrix}
			-1 & 0 & 0\\
			0 & -1 & 0\\
			0 & 0 & -1\\
		\end{pmatrix},
	$$
	and for $\P^1\times\P^1\times\P^1/\boldsymbol{\mu}_2$ the group $G$ can be generated by
	$$
		\begin{pmatrix}
			0 & 1 & 0\\
			0 & 1 & -1\\
			-1 & 1 & 0\\
		\end{pmatrix},\,
		\begin{pmatrix}
			0 & 0 & -1\\
			0 & -1 & 0\\
			-1 & 0 & 0\\
		\end{pmatrix},\,
		\begin{pmatrix}
			-1 & 0 & 0\\
			0 & -1 & 0\\
			0 & 0 & -1\\
		\end{pmatrix}.
	$$
	For both cases $G\cong \S_4 \times \CC_2$.
	\end{scase}
	
	\begin{scase}{\sc $v(P)=8$.} 
	Then $P$ combinatorially is a cube. Such polytope does exist, it is unique and was considered in Lemma \ref{lem} Case \ref{cube}. This polytope also admits the desired action for the group $G\cong\S_4 \times \CC_2$ generated by
  		 $$
			\begin{pmatrix}
			0 & 1 & -1\\
			-1 & 1 & 0\\
			0 & 1 & 0\\
			\end{pmatrix},\,
			\begin{pmatrix}
			1 & -1 & 0\\
			0 & -1 & 0\\
			0 & -1 & -1\\
			\end{pmatrix},\,
			\begin{pmatrix}
			-1 & 0 & 0\\
			0 & -1 & 0\\
			0 & 0 & -1\\
		   \end{pmatrix}.
		$$
	\end{scase}
	
	\begin{scase}{\sc $v(P)=12$.} 
	Then $P$ combinatorially is an icosahedron. But terminal Fano icosahedron does not exist by \cite{GRD}.
	\end{scase}
	\end{case}
	
	\begin{case}{\sc  $P$ combinatorially is an Archimedean solid.}
	By Proposition \ref{arch} there exist only two Archimedean solids with $6$, $8$, $12$ vertices: the truncated tetrahedron and the cuboctahedron. 
	
	The truncated tetrahedron cannot be a terminal Fano polytope because it has  hexagon facets.
	
	 By \cite{GRD} there exists only one terminal Fano cuboctahedron with the following vertices: 
	 \begin{equation*}
	 	\begin{split}
	 		\pm (1, 0, 0),\, \pm (0, 1, 0),\, \pm (0, 0, 1),\, \pm (1, 1, 0),\,\pm (1, 0, 1),\, \pm (0, -1, 1).\end{split}
	 \end{equation*}
	This polytope admits the desired group action for $G\cong\S_4\times \CC_2$ which can be generated by
	$$
		\begin{pmatrix}
			0 & 0 & 1\\
			-1 & 1 & 1\\
			0 & -1 & 0\\
		\end{pmatrix},\,
		\begin{pmatrix}
			0 & 0 & 1\\
			0 & -1 & 0\\
			1 & 0 & 0\\
		\end{pmatrix},\,
		\begin{pmatrix}
			-1 & 0 & 0\\
			0 & -1 & 0\\
			0 & 0 & -1\\
		\end{pmatrix}.
	$$
	 It determines the singular divisor of type $(1,1,1,1)$ in $(\P^1)^4$ (number $625$ in \cite{GRD}). If we assume that $(x_i:y_i)$ are homogeneous coordinates on the $i$-th factor of $(\P^1)^4$ for $i=1,\ldots,4$, then this variety can be defined by the following equation
	 $$
	 x_1x_2x_3x_4-y_1y_2y_3y_4=0
	 $$ 
	\end{case}
	
	\begin{case}{\sc  $P$ combinatorially is a prism.}
	By Lemma \ref{polygons} we should consider only the prisms with the triangular and quadrangular bases. 
	
	We already know that there is only one prism with the triangular bases which appeared in Lemma \ref{lem}. However, now this prism does not admit the group actions satisfying our conditions. The case of the prism with quadrangular bases was considered above.	
	\end{case}

	\begin{case}{\sc  $P$ combinatorially is an antiprism.}
	Similarly we should consider only the cases of the antiprism with the triangular and the quadrangular bases. The triangular base case is exactly the case of the octahedron. And there are no terminal Fano antiprisms with quadrangular bases by \cite{GRD}.
	\end{case}
	\end{proof}

Now the main result follows from Lemma \ref{lemma1}, Lemma \ref{lemma}, Lemma \ref{lem} and Lemma \ref{gap}.

\bibliography{References}
\bibliographystyle{alpha}	

\end{document}